\documentclass{amsart}
\usepackage{amsfonts,latexsym,graphicx,amssymb,amsmath,url,amsthm}
\newtheorem{theorem}{Theorem}
\newtheorem{lemma}{Lemma}

\newtheorem{conjecture*}{Conjecture}
\theoremstyle{definition}
\theoremstyle{remark}

\newcommand{\N}{\mathbb N}
\newcommand{\Z}{\mathbb Z}

\begin{document}

\title{Two dimensional arithmetic progressions avoiding squares}
\iftrue
\author[R. Dietmann] {Rainer Dietmann}
\address{Department of Mathematics, Royal Holloway,
University of London, Egham, Surrey, TW20 0EX, United Kingdom}
\email{rainer.dietmann@rhul.ac.uk}
\author[C. Elsholtz]{Christian Elsholtz}
\address{Institute of Analysis and Number Theory,
Graz University of Technology,
Kopernikusgasse 24/II, Graz,
A-8010 Graz, Austria}
\email{elsholtz@math.tugraz.at}
\thanks{Research of the second author was partially supported by the  joint FWF-ANR project ArithRand (I 4945-N and ANR-20-CE91-0006).}

\fi
\begin{abstract}
We show that any proper symmetric two dimensional arithmetic
progression contained in the interval $[-T,T]$ which avoids
non-zero perfect squares has at most
$O_\varepsilon(T^{20/27+\varepsilon})$ elements. This improves
on a result of Croot, Lyall and Rice. We also discuss
lower bounds for this problem and their connections to
bounds for the least quadratic non-residue modulo a prime.
\end{abstract}
\maketitle
\section{Introduction}

There are many questions related to sumsets and squares, or arithmetic progressions and squares.
See for example Bombieri, Granville and Pintz \cite{Bombieri-Granville-Pintz}, Bombieri and Zannier \cite{Bombieri-Zannier} or Dujella and
\iftrue
the second author 
\fi
\cite{Dujella-Elsholtz}, Gyarmati \cite{Gyarmati}, Lagarias, Odlyzko and Shearer \cite{Lagarias-etal}, Khalfallah, Lodha and Szemer\'{e}di \cite{Khalfallah-etal}, Chase \cite{Chase}, Schlage-Puchta \cite{Schlage-Puchta:2010},
\iftrue
the second author
\fi
and Wurzinger \cite{EW},
\iftrue
the second author
\fi
and Ruzsa and Wurzinger \cite{EW2}.

Croot, Lyall and Rice \cite{Croot-Lyall-Rice} raised a new type of question, after briefly discussing the
one-dimensional case (see Lemma \ref{losteria}):
Let $q_1, q_2$ be positive integers, where $q_2$ is prime, let $X_1>0, X_2>0$, and suppose that the following symmetric
two-dimensional arithmetic progression
\[
  {\mathcal A}_{q_1,q_2}(X_1,X_2)=\{ x_1 q_1+x_2 q_2\mid x_1, x_2 \in \Z,
  |x_1|\leq X_1, |x_2|\leq X_2\}\subseteq[-T,T]
\]
does not contain any non-zero perfect square and is proper, i.e. each element in ${\mathcal A}_{q_1,q_2}(X_1,X_2)$ is
represented uniquely, so $|{\mathcal A}_{q_1,q_2}(X_1,X_2)|
= (2 \lfloor X_1 \rfloor+1)(2 \lfloor X_2 \rfloor + 1)$. Also assume that $q_2 X_2 \geq q_1 X_1$.
Can we give a non-trivial upper bound on 
$|{\mathcal A}_{q_1,q_2}(X_1,X_2)|$, as a function of $T$?
Their paper answers this with
\[
  |{\mathcal A}_{q_1,q_2}(X_1,X_2)| \ll T^{5/6}(\log T)^{1/3}.
\]
They also considerably generalized their setup, allowing
for higher dimensional arithmetic progressions and intersective
polynomials instead of squares,
so that the result above was their first `model case'. They give heuristic reasons that
`properness' should not be a strong hypothesis.

In this paper we concentrate on this one case and aim to get
a bound as sharp as possible. Whereas Croot, Lyall and Rice
rely on Fourier analysis and Weyl's inequality, we follow
a more elementary approach inspired by Zaharescu's work (see \cite{Zaharescu})
and based on Burgess type bounds for character sums.
It performs better in this specific situation but is
unlikely to generalize to the more general settings considered by Croot, Lyall and Rice.

\begin{theorem}
\label{egham}
Let $q_1, q_2$ be coprime positive integers, let $X_1, X_2>0$, and suppose that
\[
  {\mathcal A}_{q_1,q_2}(X_1,X_2)=\{ x_1 q_1+x_2 q_2\mid x_1, x_2 \in \Z,
  |x_1|\leq X_1, |x_2|\leq X_2\}\subseteq[-T,T]
\]
does not contain a non-zero perfect square. Then
\[
  |{\mathcal A}_{q_1,q_2}(X_1,X_2)| \ll_\varepsilon T^{20/27+\varepsilon}.
\]
The implied $O$-constant only depends on $\varepsilon$.
(Note that $20/27= 0.74074\ldots)$. \end{theorem}

The assumptions in Theorem \ref{egham} are slightly different from the one in the
result by Croot, Lyall and Rice mentioned above, but the
following shows that our method also works in their setting: instead of assuming coprimality of $q_1$ and $q_2$ we can
alternatively assume properness of the two-dimensional arithmetic
progression.
\begin{theorem}
\label{burger}
Let $q_1, q_2$ be positive integers, let $X_1,X_2>0$, and
let $\mathcal{A}_{q_1,q_2}(X_1,X_2) \subset [-T,T]$ be a \emph{proper}
symmetric two dimensional  arithmetic progression which does not contain a non-zero perfect square. Then
\[
  |\mathcal{A}_{q_1,q_2}(X_1,X_2)| \ll_\varepsilon T^{20/27+\varepsilon}.
\]
The implied $O$-constant only depends on $\varepsilon$.
\end{theorem}

One might also ask about lower bounds for this problem.
We address this question in the last section, linking it
to the size of the smallest quadratic non-residue modulo
a prime. In particular, we will show that in Theorem
\ref{egham} and Theorem \ref{burger} no exponent smaller than
$\frac12$ is possible, and that going below $\frac12+\frac{1}
{8\sqrt{e}} \approx 0.5758$ would need an improvement on
Burgess's bound for the smallest quadratic non-residue for
primes congruent to $1$ modulo $4$.

\bigskip
\emph{Notation:} As usual, $\varepsilon$ denotes any fixed positive small number. In statements involving $\ll$ it is understood that the implied $O$-constant is allowed to depend on $\varepsilon$ unless otherwise mentioned.

\section{Preparations}

\begin{lemma}
\label{losteria}
Let $q \in \N$ and $X>0$, and suppose that
\[
  \mathcal{A}_q(X) = \{xq \mid x \in \Z, |x| \le X\}
  \subset [-T,T]
\]
does not contain a non-zero perfect square. Then
\[
  |\mathcal{A}_q(X)| \ll T^{1/2}.
\]
\end{lemma}
\begin{proof}
See the introduction of \cite{Croot-Lyall-Rice}.
\end{proof}
\begin{lemma}
\label{graz}
Let $\varepsilon>0$ and let $1 \le q_1 \le q_2$ be coprime integers.
For any $N \in \N$ there exist $x_1, x_2 \in \Z$ such that
$x_1 q_1 + x_2 q_2 = n^2$ for some integer $n$ with $1 \le n \le N$
and
\[
  x_i \ll q_2^{\frac{1}{2}+\varepsilon} + \frac{q_2^{2+\varepsilon}}{N^2}
  \quad (1 \le i \le 2).
\]
\end{lemma}
\begin{proof}
This is Lemma 4 in \cite{Zaharescu}.
\end{proof}
\begin{lemma}
\label{sirene}
Let $q_1, q_2 \in \N$ be coprime, and let $N \in \N$. Then there exist
$x_1, x_2 \in  \Z$ such that
\[
  |x_1| \ll \frac{N^2}{q_1}+\frac{q_1^{5/4+\varepsilon} q_2}{N^2},
  \quad
  |x_2| \ll \frac{q_1^{9/4+\varepsilon}}{N^2}
\]
and $x_1 q_1+x_2 q_2$ is a square $n^2$ with $1 \le n^2 \le N^2$.
\end{lemma}
\begin{proof}
As in \cite[p. 380]{Zaharescu}, for each $k \in \N$ we define
$H(k)$ to be the smallest $h  \in \N$ such that for all
integers $x$ coprime to $k$ there exist integers $y,z$ with
$|y| \le h$ and $x \equiv y z^2 \pmod k$.
Following the proof of Lemma 3 in \cite{Zaharescu} we find
integers $b, c$ with $|b| \le H(q_1)$ and $b$ coprime to $q_1$ such that $bq_2 \equiv c^2 \pmod {q_1}$. Note
that $H(q_1) \ll q_1^{1/4+\varepsilon}$ by Lemma 1 in
\cite{Zaharescu}. Clearly $c$ is coprime to $q_1$ because $b$ and $q_2$ are. Let $\overline{c}$ be an integer that is a
multiplicative inverse of $c$, i.e. $c \overline{c} \equiv 1
\pmod {q_1}$. Further following the proof of Lemma 3 in
\cite{Zaharescu}, using Dirichlet's theorem in Diophantine
approximation we find an integer $n$ with $1 \le n \le N$
such that the distance of $\frac{n\overline{c}}{q_1}$ to the
nearest integer is at most $\frac{1}{N}$. We can then choose
an integer $m$ such that
\[
  \left| \frac{n\overline{c}+mq_1}{q_1} \right| \le \frac1N.
\]
Writing $d=n\overline{c}+mq_1$, we find that $d \equiv
n\overline{c} \pmod {q_1}$ and $|d| \le \frac{q_1}{N}$.
As in \cite{Zaharescu}, it follows that
\[
  n^2 \equiv d^2 c^2 \equiv b q_2 d^2 \pmod {q_1}.
\]
Hence
\[
  x_1=\frac{n^2-bq_2d^2}{q_1}
\]
is an integer. Putting $x_2=bd^2$, we have found
integers $x_1, x_2$ such that $1 \le x_1 q_1 + x_2 q_2=n^2 \le N^2$,
$x_1 q_1 + x_2 q_2=n^2$ is a square, $|x_2| \ll H(q_1) q_1^2 N^{-2}
\ll q_1^{9/4+\varepsilon} N^{-2}$ and
\[
  |x_1| \le \frac{N^2}{q_1} + \frac{q_2}{q_1} |x_2| \ll
  \frac{N^2}{q_1} + \frac{q_1^{5/4+\varepsilon} q_2}{N^2}
\]
as required.
\end{proof}

In the following suppose that $q_1, q_2$ are coprime integers with
$1 \le q_1 \le q_2$, and that $X_1, X_2, T$ are positive numbers,
such that
\[
  {\mathcal A}_{q_1,q_2}(X_1,X_2)=\{ x_1 q_1+x_2 q_2\mid x_1, x_2 \in \Z,
  |x_1|\leq X_1, |x_2|\leq X_2\}\subseteq[-T,T]
\]
and no element of $\mathcal{A}_{q_1,q_2}(X_1,X_2)$ is a non-zero square. We clearly can assume that $X_1 \ge 1$
and $X_2 \ge 1$: If, say $X_1<1$, then necessarily $x_1=0$, and we are
reduced to a one-dimensional progression for which Lemma \ref{losteria} provides
an even better bound than required in Theorem \ref{egham}
or Theorem \ref{burger}. Analogously, for $X_2<1$.

\begin{lemma}
\label{abelsch}
We have
\[
  X_i \ll \frac{T}{q_i} \quad (1 \le i \le 2).
\]
\end{lemma}
\begin{proof}
This follows trivially from $\mathcal{A}_{q_1,q_2}(X_1, X_2) \subset [-T,T]$.
\end{proof}

\begin{lemma}
\label{laber}
We have
\[
  X_i \ll T^{1/2} \quad (1 \le i \le 2).
\]
\end{lemma}
\begin{proof}
Without loss of generality let $i=1$. Put $x_2=0$. Then
$\mathcal{A}_{q_1}(X_1)$ is a one-dimensional arithmetic progression
in $[-T,T]$ avoiding non-zero squares, whence by Lemma
\ref{losteria} we have $X_1 \ll T^{1/2}$.
\end{proof}

\begin{lemma}
\label{denmark}
We have $X_1 \ll q_2^{\frac{1}{2}+\varepsilon}$ or
$X_2 \ll q_2^{\frac{1}{2}+\varepsilon}$.
\end{lemma}
\begin{proof}
Choose $N=q_2^{3/4}$ in Lemma \ref{graz}. If
$X_1 \gg q_2^{1/2+\varepsilon}$ and
$X_2 \gg q_2^{1/2+\varepsilon}$, then there exist integers
$x_1, x_2$ with $|x_1| \le X_1$ and $|x_2| \le X_2$ such that
$x_1 q_1 + x_2 q_2$ is a non-zero square, which is a contradiction.
\end{proof}

\begin{lemma}
\label{dachterrasse}
For every $N \in \N$ we have
\[
  X_1 \ll \frac{N^2}{q_1}+\frac{q_1^{5/4+\varepsilon} q_2}{N^2}
  \quad \text{or} \quad
  X_2 \ll \frac{q_1^{9/4+\varepsilon}}{N^2}.
\]
\end{lemma}
\begin{proof}
This follows from Lemma \ref{sirene}.
\end{proof}

\begin{lemma}
\label{bodensee}
For every $N \in \N$ we have
\[
  X_1 \ll \frac{q_2^{9/4+\varepsilon}}{N^2}
  \quad \text{or} \quad
  X_2 \ll \frac{N^2}{q_2}+\frac{q_2^{5/4+\varepsilon} q_1}{N^2}.
\]
\end{lemma}
\begin{proof}
Reverse the roles of $x_1, q_1$ versus $x_2, q_2$ in Lemma \ref{sirene} and
Lemma \ref{dachterrasse}.
\end{proof}

\begin{lemma}
\label{schwarzenegger}
We have
\[
  |\mathcal{A}_{q_1,q_2}(X_1,X_2)| \ll \frac{q_1^{1/8+\varepsilon} T}{q_2^{1/2}}.
\]
\end{lemma}
\begin{proof}
Choose $N=q_1^{9/16} q_2^{1/4}$ in Lemma \ref{dachterrasse}. Then
$X_1 \ll q_1^{1/8+\varepsilon} q_2^{1/2}$ or
$X_2 \ll q_1^{9/8+\varepsilon} q_2^{-1/2}$.
In the first case, we use the bound $X_2 \ll \frac{T}{q_2}$ from Lemma
\ref{abelsch} and obtain
\[
  |\mathcal{A}_{q_1, q_2}(X_1,X_2)| \ll X_1 X_2 \ll
  q_1^{1/8+\varepsilon} q_2^{-1/2} T.
\]
In the second case, we use the bound $X_1 \ll \frac{T}{q_1}$
from Lemma \ref{abelsch}, and obtain
\[
  |\mathcal{A}_{q_1, q_2}(X_1,X_2)| \ll X_1 X_2 \ll T q_1^{1/8+\varepsilon} q_2^{-1/2}
\]
as well.
\end{proof}

\section{Proof of Theorem \ref{egham}}
By Lemma \ref{denmark} we have $X_1 \ll q_2^{1/2+\varepsilon}$ or
$X_2 \ll q_2^{1/2+\varepsilon}$.\\

\smallskip
\noindent
\textbf{Case I:} $X_1 \ll q_2^{1/2+\varepsilon}$.\\

\smallskip
\noindent
\textbf{Subcase A:} $q_2 \le T^{4/7}$.

\smallskip
\noindent
We choose $N \in \N$ with
\[
  N^2 \asymp \frac{q_2^{9/4+\varepsilon}}{X_1}.
\]
Then by Lemma \ref{bodensee} and $q_1 \le q_2 \le T^{4/7}$ we have 
\[
  X_2 \ll \frac{q_2^{5/4+\varepsilon}}{X_1} + \frac{q_1 X_1}{q_2}
  \ll \frac{T^{5/7+\varepsilon}}{X_1} + X_1.
\]
Moreover,
\[
  X_1^2 \ll q_2^{1+2\varepsilon} \ll T^{4/7+2\varepsilon}.
\]
Hence
\[
  |\mathcal{A}_{q_1,q_2}(X_1,X_2)| \ll X_1 X_2 \ll T^{5/7+\varepsilon}.
\]

\smallskip
\noindent
\textbf{Subcase B:} $q_2>T^{4/7}$.

\smallskip
\noindent
By Lemma \ref{abelsch} we have $X_2 \ll \frac{T}{q_2}$, so
\[
  |\mathcal{A}_{q_1,q_2}(X_1,X_2)| \ll X_1 X_2 \ll
  q_2^{1/2+\varepsilon} \frac{T}{q_2} \ll T^{5/7+\varepsilon}.
\]

\smallskip
\noindent
\textbf{Case II:} $X_2 \ll q_2^{1/2+\varepsilon}$.\\

\smallskip
\noindent
\textbf{Subcase A:} $q_2 \ge T^{2/3}$:

\smallskip
\noindent
\textbf{Subsubcase 1:} $q_1 \le T^{16/27}$:
Then from Lemma \ref{schwarzenegger} we obtain
\[
  |\mathcal{A}_{q_1, q_2}(X_1,X_2)| \ll \frac{q_1^{1/8+\varepsilon} T}{q_2^{1/2}}
  \ll T^{20/27+\varepsilon}
\]
as required.

\smallskip
\noindent
\textbf{Subsubcase 2:} $q_1>T^{16/27}$: Here we use Lemma \ref{abelsch}
and again obtain
\[
  |\mathcal{A}_{q_1, q_2}(X_1,X_2)| \ll X_1 X_2 \ll \frac{T^2}{q_1 q_2}
  \ll T^{20/27}.
\]

\smallskip
\noindent
\textbf{Subcase B:} $q_2<T^{2/3}$:

\smallskip
\noindent
Write
\[
  q_1=T^a, \quad q_2=T^b, \quad a \le b \le \frac23.
\]

\smallskip
\noindent
\textbf{Subsubcase 1:} $a+2b \le \frac{52}{27}$.

\smallskip
\noindent
By Lemma \ref{dachterrasse}, for each $N \in \N$ we either have
$X_1 \ll \frac{N^2}{q_1}+\frac{q_1^{5/4+\varepsilon}q_2}{N^2}$, in which
case we use $X_2 \ll q_2^{1/2+\varepsilon}$ by assumption of case II, or
$X_2 \ll \frac{q_1^{9/4+\varepsilon}}{N^2}$, in which case we use
$X_1 \ll \frac{T}{q_1}$ from Lemma \ref{abelsch}.
Hence
\[
  |\mathcal{A}_{q_1, q_2}(X_1,X_2)| \le X_1 X_2
  \ll \min_{N \in \N} \max
  \left\{ \left( \frac{N^2}{q_1}+\frac{q_1^{5/4+\varepsilon}q_2}{N^2}
  \right) q_2^{1/2+\varepsilon},
  \frac{q_1^{9/4+\varepsilon}}{N^2} \frac{T}{q_1} \right\}.
\]
We now show that it is possible to choose $N \in \N$ such that
simultaneously
\begin{align}
  \frac{N^2}{q_1} q_2^{1/2+\varepsilon}
  & \le T^{20/27+2\varepsilon}, \label{cond1}\\
  \frac{q_1^{5/4+\varepsilon}q_2}{N^2} q_2^{1/2+\varepsilon} & \le
  T^{20/27+\varepsilon},
  \label{cond2}\\
  \frac{q_1^{9/4+\varepsilon}}{N^2} \cdot \frac{T}{q_1} & \le
  T^{20/27},
  \label{cond3}
\end{align}
from which $|\mathcal{A}_{q_1,q_2}(X_1, X_2)| \ll T^{20/27+2\varepsilon}$ follows.
Condition \eqref{cond1} gives the upper bound
\begin{equation}
\label{cond4}
  N^2 \le q_1 q_2^{-1/2-\varepsilon} T^{20/27+2\varepsilon},
\end{equation}
condition \eqref{cond2} gives the lower bound
\begin{equation}
\label{cond5}
  N^2 \ge q_1^{5/4+\varepsilon} q_2^{3/2+\varepsilon}
  T^{-20/27-\varepsilon},
\end{equation}
and condition \eqref{cond3} gives the lower bound
\begin{equation}
\label{cond6}
  N^2 \ge q_1^{5/4+\varepsilon} T^{7/27}.
\end{equation}
Note that
\[
  q_1^{5/4+\varepsilon} q_2^{3/2+\varepsilon} T^{-20/27-\varepsilon} \le
  q_1^{5/4+\varepsilon} T^{7/27}
\]
as
$q_2 \le T^{2/3}$. To check that conditions \eqref{cond4}, \eqref{cond5}
and \eqref{cond6} are compatible it is therefore sufficient to check that
\eqref{cond4} and \eqref{cond6} are compatible. Now
\[
  q_1^{5/4+\varepsilon} T^{7/27} \le q_1 q_2^{-1/2-\varepsilon}
  T^{20/27+2\varepsilon}
\]
is equivalent to
\[
  q_1^{1/4+\varepsilon} q_2^{1/2+\varepsilon} \le T^{13/27+2\varepsilon},
\]
which is satisfied as $a+2b \le \frac{52}{27}$.
So we can indeed choose a positive integer $N$ satisfying
\eqref{cond1}--\eqref{cond3}, which shows that
\[
  |\mathcal{A}_{q_1, q_2}(X_1,X_2)| \ll T^{20/27+2\varepsilon}
\]
in this case.

\smallskip
\noindent
\textbf{Subsubcase 2:} $a+2b>\frac{52}{27}$.

\smallskip
\noindent
In this case, we use the bound $X_1 \ll \frac{T}{q_1}$ due to Lemma
\ref{abelsch} along with $X_2 \ll q_2^{1/2+\varepsilon}$ and $b \le \frac 23$ to arrive at
\begin{align*}
  |\mathcal{A}_{q_1, q_2}(X_1,X_2)| & \ll X_1 X_2 \ll \frac{T}{q_1}
  q_2^{1/2+\varepsilon} \ll T^{1-a+b/2+\varepsilon}
  \ll T^{-\frac{25}{27}+\frac{5b}{2}+\varepsilon} \\
  & \ll T^{20/27+\varepsilon}.
\end{align*}

\section{Proof of Theorem \ref{burger}}
We can reduce Theorem \ref{burger} to Theorem \ref{egham}
by getting rid of the greatest common divisor between $q_1$
and $q_2$ and applying a recursive reduction strategy:
We may clearly assume that
$T$ is an integer, and by allowing for a sufficiently large $O$-constant,
that $T \ge 10$. The proof uses a recursive argument in one case and
in fact amounts to induction on $T$. So fix some integer $T \ge 10$
and suppose
that the result has already been established for all
$\mathcal{A}_{q_1,q_2} (X_1,X_2)
\subset [-W,W]$ for integers $W<T$. Now let $\mathcal{A}_{q_1,q_2}(X_1,X_2)
\subset [-T,T]$ be a proper two dimensional generalized arithmetic progression
not containing any non-zero perfect square.
Write $d=(q_1,q_2)$ and
\[
  \tilde{q}_1=\frac{q_1}{d}, \quad \tilde{q}_2=\frac{q_2}{d}.
\]
Clearly, $(\tilde{q}_1,\tilde{q}_2)=1$.
Recall that without loss of generality we can assume that
$X_1 \ge 1$ and $X_2 \ge 1$.

\bigskip
\noindent
\textbf{Case I:} $d \ll 1$, for a sufficiently large implied absolute
$O$-constant. Then if $X_1 \le d$ or $X_2 \le d$, using $X_2 \ll T^{1/2}$
or $X_1 \ll T^{1/2}$, respectively (see Lemma \ref{laber}), we immediately obtain
$X_1 X_2 \ll T^{1/2}$, whence $|\mathcal{A}_{q_1,q_2}(X_1,X_2)| \ll
T^{1/2}$, which is even better than required. If both $X_1 \ge d$
and $X_2 \ge d$, let
\[
  \tilde{X}_1=\frac{X_1}{d}, \quad \tilde{X}_2=\frac{X_2}{d},
\]
and substitute
\[
  x_1=d a_1, \quad x_2=d a_2
\]
for integers $|a_1| \le \tilde{X}_1, |a_2| \le \tilde{X}_2$. Writing
\[
  x_1 q_1 + x_2 q_2 = d^2 (a_1 \tilde{q}_1 + a_2 \tilde{q}_2),
\]
we see that $\mathcal{A}_{\tilde{q}_1,\tilde{q}_2}
(\tilde{X}_1,\tilde{X}_2) \subset
[-T/d^2, T/d^2]$ contains no non-zero perfect square. Hence, 
as $(\tilde{q}_1, \tilde{q}_2)=1$, by
Theorem \ref{egham} we obtain
\[
  \tilde{X}_1 \tilde{X}_2 \ll T^{20/27+\varepsilon},
\]
whence
\[
  |\mathcal{A}_{q_1,q_2}(X_1,X_2)| \ll X_1 X_2 \ll d^2
  \tilde{X}_1 \tilde{X}_2 \ll T^{20/27+\varepsilon}.
\]
\bigskip
\noindent
\textbf{Case II:} $d \ge T^{1/2}$.\\
If both $X_1 \ge \tilde{q}_2$ and $X_2 \ge \tilde{q}_1$, then
$x_1 \tilde{q}_1 + x_2 \tilde{q}_2$ can take the same value for
different $x_1, x_2$ with $|x_1| \le X_1, |x_2| \le X_2$, contradicting
our assumption of properness. Hence
$X_1 \le \tilde{q}_2$ or $X_2 \le \tilde{q}_1$. In the first case, by
Lemma \ref{abelsch} we obtain $X_2 \ll \frac{T}{q_2}=\frac{T}{d\tilde{q}_2}$, thus
$X_1 X_2 \ll \frac{T}{d} \ll T^{1/2}$. The second case is analogous.
Hence $|\mathcal{A}_{q_1,q_2}(X_1,X_2)| \ll X_1 X_2 \ll T^{1/2}$ which
is even better than required.

\bigskip
\noindent
\textbf{Case III:} $1 \ll d \le T^{1/2}$.
Without loss of generality we may assume that $X_1 \le X_2$. Write
\[
  U = \frac{X_2}{X_1}.
\]
Like in Case I, we will use a reduction argument, eventually
splitting off a square $d^2$ and reducing to an analogous problem
in a shorter interval $[-T/d^2, T/d^2]$: Let $C$ be the centrally symmetric
convex body
\[
  C = \{\mathbf{x} \in \mathbb{R}^2: |x_1| \le U^{-1/2}, |x_2| \le U^{1/2}\}
\]
of volume $\operatorname{vol}(C)=4$. Further, let $\lambda_1 \le \lambda_2$
be the successive minima of $C$ with respect to the lattice
\begin{equation}
\label{lattice}
  L = \{\mathbf{x} \in \mathbb{Z}^2: x_1 \tilde{q}_1 + x_2 \tilde{q}_2
  \equiv 0 \pmod d\}.
\end{equation}
Notice that $(\tilde{q}_1,\tilde{q}_2,d)=1$ as
$(\tilde{q}_1,\tilde{q}_2)=1$,
so $L$ is a two-dimensional lattice of determinant $\det L=d$. Hence,
by Minkowski's second theorem in the geometry of numbers,
\begin{equation}
\label{margaretenbad}
  d \ll \lambda_1 \lambda_2 \ll d.
\end{equation}
Let
\begin{equation}
\label{muecken}
  \tilde{X}_1 = \frac{X_1 U^{1/2}}{2 \lambda_1}, \quad
  \tilde{X}_2 = \frac{X_1 U^{1/2}}{2 \lambda_2}.
\end{equation}
By definition  of $C$ and $\lambda_1,\lambda_2$, there exist
linearly independent
\begin{equation}
\label{eistee}
  \begin{pmatrix} u_1\\ u_2 \end{pmatrix},
  \begin{pmatrix} v_1\\ v_2 \end{pmatrix} \in L
\end{equation}
such that
\begin{eqnarray*}
  |u_1| & \le \lambda_1 U^{-1/2}, \quad |u_2| \le \lambda_1 U^{1/2},\\
  |v_1| & \le \lambda_2 U^{-1/2}, \quad |v_2| \le \lambda_2 U^{1/2}.
\end{eqnarray*}
Now substitute
\begin{equation}
\label{bodensee2}
  \begin{pmatrix} x_1\\x_2 \end{pmatrix} =
  a_1 \begin{pmatrix} u_1\\u_2 \end{pmatrix} +
  a_2 \begin{pmatrix} v_1\\v_2 \end{pmatrix}
\end{equation}
for integers $|a_1| \le \tilde{X}_1, |a_2| \le \tilde{X}_2$, whence
\begin{align*}
  |x_1| & \le |a_1| |u_1| +
  |a_2| |v_1| \\ & \le \frac12 X_1 U^{1/2} \lambda_1^{-1} \lambda_1 U^{-1/2}
  + \frac12 X_1 U^{1/2} \lambda_2^{-1} \lambda_2 U^{-1/2} \le X_1,\\
  |x_2| & \le |a_1| |u_2| + |a_2| |v_2| \\
  & \le \frac12 X_1 U^{1/2} \lambda_1^{-1} \lambda_1 U^{1/2}
  + \frac12 X_1 U^{1/2} \lambda_2^{-1} \lambda_2 U^{1/2} \le UX_1 = X_2.
\end{align*}
Now from \eqref{lattice} and \eqref{eistee} we obtain
\[
  u_1 \tilde{q}_1 + u_2 \tilde{q}_2 = d p_1, \quad
  v_1 \tilde{q}_1 + v_2 \tilde{q}_2 = d p_2
\]
for certain integers $p_1, p_2$. Using the substitution \eqref{bodensee2},
we get
\[
  x_1 q_1 + x_2 q_2 = d ((a_1 u_1 + a_2 v_1) \tilde{q}_1
  + (a_1 u_2 + a_2 v_2) \tilde{q}_2)
  = d^2 (a_1 p_1 + a_2 p_2).
\]
As shown above, for $|a_1| \le \tilde{X}_1$ and $|a_2| \le
\tilde{X}_2$ we have $|x_1| \le X_1$, $|x_2| \le X_2$, whence for
$|a_1| \le \tilde{X}_1, |a_2| \le \tilde{X}_2$ the number
$d^2 (a_1 p_1 + a_2 p_2)$ cannot be a non-zero perfect square, as
$\mathcal{A}_{q_1, q_2}(X_1, X_2)$ does not contain any non-zero perfect
square. Hence for $|\alpha_1| \le \tilde{X}_1, |\alpha_2| \le \tilde{X}_2$
the number $\alpha_1 p_1 + \alpha_2 p_2$ cannot be a non-zero perfect
square. We conclude that
\[
  \mathcal{A}_{p_1, p_2}(\tilde{X}_1, \tilde{X}_2) \subset
  \left[ -\frac{T}{d^2}, \frac{T}{d^2} \right]
\]
is a symmetric two dimensional arithmetic progression not containing
any non-zero perfect square. It is also obvious from the
bijective substitution \eqref{bodensee2} that it is a proper two dimensional
arithmetic progression as $\mathcal{A}_{q_1, q_2}(X_1, X_2)$ is.
Now $d \ge 2$ so $\frac{T}{d^2} \le \frac{T}{4}$, so
\[
  \left[ \frac{T}{d^2} \right] < T.
\]
Hence, by inductive assumption,
\[
  |\mathcal{A}_{p_1, p_2}(\tilde{X}_1, \tilde{X}_2)|
  \ll \left(\frac{T}{d^2} \right)^{20/27
  +\varepsilon};
\]
note that $\frac{T}{d^2} \ge 1$.
From \eqref{margaretenbad} and \eqref{muecken} and
$\mathcal{A}_{p_1, p_2}(\tilde{X}_1, \tilde{X}_2)$ being
proper we obtain
\[
  |\mathcal{A}_{p_1, p_2}(\tilde{X}_1, \tilde{X}_2)| \gg
  \tilde{X}_1 \tilde{X}_2 \gg \frac{X_1 X_2}{d}.
\]
Note that the implied $O$-constant in the previous line is absolute.
Consequently,
\[
  X_1 X_2 \ll d^{-13/27}
  T^{20/27+\varepsilon} \ll T^{20/27+\varepsilon}.
\]
As $X_1 \ge 1$ and $X_2 \ge 1$, it follows that
\[
  |\mathcal{A}_{q_1, q_2}(X_1, X_2)| \ll X_1 X_2
  \ll T^{20/27+\varepsilon}
\]
as required. Due to the inductive form of this argument, one might wonder
whether repeated recursive steps might lead to a blowup of the implied
$O$-constant. However, as $d \gg 1$ for a sufficiently large implied absolute
$O$-constant, the saving provided by $d^{-13/27}$ makes sure
we can keep the same implied $O$-constant throughout this process.

\section{Lower bounds}
Finally, let us also study lower bounds in this problem. The discussion below links this to lower bounds on the least quadratic non-residue and shows that a lower bound of size better than $T^{\frac{1}{2}+ \delta}$ infinitely often, would require completely unexpected lower bounds on quadratic non-residues, as it would contradict Linnik's conjecture that the least quadratic non-residue is of size $O_{\epsilon}(p^{\varepsilon})$.

For prime $p$ write $n(p)$ for the smallest
quadratic non-residue modulo $p$. It is well known that
$n(p)<\sqrt{p}$ for all $p$ and that $n(p)=O(n^{1/(4\sqrt{e})+\varepsilon})$
(see Burgess \cite{Burgess:1957}).
Let $\varepsilon>0$ and $p$ be a sufficiently large prime with
$p \equiv 1 \pmod 4$. Choose $q \in \N$ with $p \le q <2p$ such that
$\left( \frac{q}{p} \right)=-1$ and let $T=2 p^2$, $X_1=p-1$ and
$X_2=n(p)-1$. Then
\[
  \mathcal{A}_{p, q}(X_1, X_2) \subset [-T,T]
\]
does not contain any non-zero square: Suppose that $x_1 p + x_2 q$ is a  non-zero square.
Then $\left( \frac{x_2 q}{p} \right) = 1$. As $\left( \frac{q}{p} \right)
=-1$, this implies that $\left( \frac{x_2}{p} \right)=-1$.
Since $p \equiv 1 \pmod 4$, we obtain $\left( \frac{-x_2}{p} \right)=-1$
as well. Hence, by definition of $n(p)$, we have $|x_2| \ge n(p) > X_2$.
Now all the elements in $\mathcal{A}_{p, q}(X_1, X_2)$ are
distinct, as $x_1 p + x_2 q = y_1 p + y_2 q$ implies that
$p \mid (x_2-y_2) q$, and therefore, as $(p,q)=1$,
$p \mid x_2-y_2$, which by virtue of $|x_2|, |y_2| \le X_2
<n(p)<\sqrt{p}$ implies that $x_2=y_2$ and thus $x_1=y_1$.
Hence
\[
  |\mathcal{A}_{p, q}(X_1,X_2)| \gg p n(p) \gg T^{1/2} n(p).
\]
Following the argument of Sali\'{e} (see \cite{Salie}, \cite{F} and see also \cite{GR} for a slightly
stronger result), it is easy to see that
there are infinitely many primes $p \equiv 1 \pmod 4$ with
$n(p) \gg \log p$. Hence there is a sequence of $T_i \rightarrow \infty$
and two-dimensional progressions $\mathcal{A}_{q_1^{(i)}, q_2^{(i)}}(
X_1^{(i)}, X_2^{(i)}) \in [-T_i,T_i]$ not containing
any non-zero square and
\[
  |\mathcal{A}_{q_1^{(i)}, q_2^{(i)}}(X_1^{(i)}, X_2^{(i)})| \gg T_i^{1/2}
  \log T_i.
\]

Moreover, if $n(p)\gg_{\alpha} p^{\alpha}$ for infinitely many primes $p\equiv 1 \bmod 4$ (in contradiction to Linnik's conjecture), then
\[
  |\mathcal{A}_{q_1, q_2}(X_1, X_2)| \gg_{\alpha} T^{\frac{1+\alpha}{2}}
\]
infinitely often. In particular, any improvement of Theorem \ref{egham} below
\[
  |\mathcal{A}_{q_1, q_2}(X_1, X_2)| \ll T^{\frac{1}{2}+1/(8 \sqrt{e})+
  \varepsilon}
\]
would imply improving the best known upper bound for the least quadratic
non-residue, at least for primes $p \equiv 1 \pmod 4$. Note that 
$\frac{1}{2}+\frac{1}{8 \sqrt{e}}\approx 0.5758$.

\section*{Acknowledgments} 
Both authors want to thank the Department of Mathematics at Royal Holloway and the Institut
f\"ur Analysis and Zahlentheorie at the TU Graz for favourable
working conditions, and they want to thank the referee for carefully reading this manuscript.


\end{document}